\documentclass[a4paper,12pt]{amsart}
\newtheorem*{thm}{Theorem}

\newtheorem{prop}{Proposition}

 \theoremstyle{remark}

\usepackage{fullpage}

\begin{document}
\title[Irrationality of $\mathbf{\zeta(3)}$]{A simplification of Ap\'ery's proof of the irrationality of $\zeta(3)$}
\author{Krishnan Rajkumar}
\thanks{Institute of Mathematical Sciences,
Chennai, India \\
\textit{E-mail address}: krishnan.rjkmr@gmail.com}
\begin{abstract}
 A simplification of Ap\'ery's proof of the irrationality of $\zeta(3)$ is presented. The construction
of approximations is motivated from the viewpoint of $2$-dimensional recurrence relations which simplifies
 many of the details of the proof.
 Conclusive evidence is also presented that these constructions arise from a continued fraction due to 
Ramanujan. 
\end{abstract}
\maketitle
\section{Introduction}

 In 1978, R. Ap\'ery \cite{Ap1} presented his famous proof of the irrationality of $\zeta(3)$.
 His method involved the explicit construction of two solutions $a_n$ and $b_n$ of the recurrence
\begin{equation}\label{Ap-rec}
(n+1)^3 u_{n+1} = (34 n^3 + 51 n^2 + 27 n + 5) u_n - n^3 u_{n-1},
\end{equation}
for $n \geq 1$ such that $a_n/b_n \rightarrow \zeta(3)$ as $n\rightarrow \infty$.
These solutions also satisfied the arithmetic properties $b_n, [1,2,\ldots n]^3 a_n \in 
\mathbb{Z}$. Put together, these facts turned out to be sufficient to complete the proof of irrationality of
 $\zeta(3)$. For an account of the history and the ``miraculous'' nature of this construction,
 see van der Poorten \cite{vdP}.

\medskip

Many proofs of the irrationality of $\zeta(3)$ have followed since then, all of which construct
the same sequences, $a_n$ and $b_n$, by vastly different methods (see Fischler \cite{Fis} for a survey).
One of these proofs is by Ap\'ery himself \cite{Ap2} (arguably his only complete proof of this
result). This paper deals with a method of interpolation for continued fractions and constructs a series of continued
fractions for $\zeta(3)$ from which the sequences $a_n$ and $b_n$ are obtained. We also remark that none
 of these proofs have generalisations to higher zeta values. For example, it is still unknown whether 
$\zeta(5)$ is irrational. 

\medskip

In Vol. 2 of Ramanujan's notebooks, Berndt \cite{Berndt} suggests that a certain continued
fraction of Ramanujan is related to the proof in \cite{Ap2}. Recall 
that the Hurwitz zeta function, $\zeta(s,x)$ is defined for $\textrm{Re } s >1, \textrm{Re } x >0$ by
\begin{equation*}
 \zeta(s,x) = \sum_{n=0}^{\infty} \frac{1}{(n+x)^s}\cdot
\end{equation*}
The continued fraction of Ramanujan in consideration (\cite{Berndt}, Entry 32(iii), p. 153)  is
\begin{equation}\label{cont-fr}
 \zeta(3,x+1) = \frac{1}{P(0,x)+}\, \frac{-1^6}{P(1,x)+} \, \frac{-2^6}{P(2,x)+}
\, \frac{-3^6}{P(3,x)+}\cdots
\end{equation}
for $\textrm{Re } x >-\frac{1}{2}$ where $P(n,x) = n^3 + (n+1)^3 + (4n+2)x(x+1)$. In the discussion following
this entry in \cite{Berndt}, it is stated that the specialisation $x=1$ yields a continued
fraction for $\zeta(3)$ which is ``of crucial importance'' in the work of Ap\'ery \cite{Ap2}.  

\medskip

Around the same time, F. Ap\'ery \cite{Ap3} in a biographical note on his father R. Ap\'ery, states 
that the construction in \cite{Ap1} is motivated from a ``number table due to Ramanujan''.
 
\medskip

In this note, we take the view that a more detailed analysis of the method in \cite{Ap2}
leads one to the conclusion that the constructions are indeed 
based on Ramanujan's continued fraction \eqref{cont-fr}. We present here such an analysis, which 
also allows us to give a simplified proof of Ap\'ery's result, which we state as 
\begin{thm}
 $\zeta(3)$ is irrational.
\end{thm}

This note is organised as follows. We first present the constructions of \cite{Ap2} from the viewpoint
of $2$-dimensional recurrence relations in Sec. 2. Section 3 is devoted to the proof of the theorem. In Sec. 4,
 we end with some concluding remarks on the comparison with Ap\'ery's approach, the relation
 to the Ramanujan's continued 
fraction and generalisations to other constants.

\section{Construction of the tables}
We start by defining the homogenous polynomials
\begin{align}
 f(i,j) &= i^3 + 2 i^2 j + 2 i j^2 + j^3, \\
 g(i,j) &= i^3 - 2 i^2 j + 2 i j^2 - j^3. \nonumber
\end{align}
We identify here the key properties of these polynomials that will be used in 
the construction. They are
\begin{align}
 f(i,j) g(i,j) = f(i,0)&g(i,0) + f(0,j)g(0,j), \label{cond1}\\
 f(i+1,j)-f(i,j+1) &= g(i+1,j+1)-g(i,j). \label{cond2}
\end{align}

Now we present the $2$-dimensional recurrence which will play a central role in our construction.

\begin{prop}\label{prop1}
 The recurrence
\begin{equation}\label{rec1}
   \begin{pmatrix}
      f(i,j) & g(0,j) \\
      f(0,j) & g(i,j)
   \end{pmatrix}
   \begin{pmatrix}
      u_{i-1,j} \\
      u_{i-1,j-1}
   \end{pmatrix} 
= f(i,0) 
   \begin{pmatrix}
      u_{i,j} \\
      u_{i,j-1}
   \end{pmatrix}.
\end{equation}
for integers $i,j\geq 1$ has a rational valued solution $u_{i,j}$ for each of the following 
boundary conditions
\begin{itemize}
 \item[(a)] $\forall i,j\geq 0$, $u_{0,j}=u_{0,i}=1$,
\item[(b)]$u_{0,0}=0$ and $\forall i,j\geq 1$,
\begin{align*}
 u_{0,j} &= \sum_{n\leq j}\tfrac{1}{f(0,n)}\prod_{k<n}\tfrac{-g(0,k)}{f(0,k)} 
= \sum_{n\leq j} n^{-3}, \\
u_{i,0} &= \sum_{n\leq i}\tfrac{1}{f(n,0)}\prod_{k<n}\tfrac{g(k,0)}{f(k,0)} 
= \sum_{n\leq i} n^{-3}.
\end{align*}
\end{itemize}
\end{prop}
\begin{proof}
 We first derive additional conditions that any solution of \eqref{rec1} has to satisfy. The top entry
 on the right in \eqref{rec1} for $i,j$ is the same as the bottom entry 
for $i,j+1$. Hence the solution has to satisfy the recurrence
\begin{equation}\label{rec2}
 f(0,j+1) u_{i-1,j+1} = (f(i,j)-g(i,j+1)) u_{i-1,j} + g(0,j) u_{i-1,j-1}.
\end{equation}
This is a condition on the solution for row $i-1$ with $i\geq 1$.
Next by inverting \eqref{rec1} and using property \eqref{cond1}, we get
\begin{equation}\label{rec3}
   \begin{pmatrix}
      g(i,j) & -g(0,j) \\
      -f(0,j) & f(i,j)
   \end{pmatrix}
   \begin{pmatrix}
      u_{i,j} \\
      u_{i,j-1}
   \end{pmatrix} 
= g(i,0) 
   \begin{pmatrix}
      u_{i-1,j} \\
      u_{i-1,j-1}
   \end{pmatrix}.
\end{equation}
This equation likewise leads to the condition that $u_{i,j}$ satisfies
\begin{equation}\label{rec4}
 f(0,j+1) u_{i,j+1} = (f(i,j+1)-g(i,j)) u_{i,j} + g(0,j) u_{i,j-1}.
\end{equation}
This is a condition on the solution for row $i$, which, by property \eqref{cond2}, 
is the same as \eqref{rec2} with $i-1$ replaced by $i$.

\medskip

Conversely, the recurrence \eqref{rec1} can be used to construct row $i$ from row $i-1$ in a 
well-defined manner, if 
 \eqref{rec2} is satisfied for row $i-1$. We will use this observation, inductively along the rows,
 to construct the required solutions.

\medskip

 First it can be easily verified that both of the given boundary conditions (a) and (b) 
satisfy \eqref{rec2} for the row $i=0$. This will be the base case. 

We now assume that we have constructed the solution upto row $i-1$ and 
that the row $i-1$ satisfies \eqref{rec2}. This implies that the recurrence \eqref{rec1}
can be used to construct row $i$ from row $i-1$ in a well-defined manner. Hence \eqref{rec3} holds
and by the discussion above, we can then conclude that row $i$ also satisfies
\eqref{rec2} and the induction step is complete.

\medskip

Hence there exist solutions to \eqref{rec1} satisfying the given boundary conditions (a) and (b) along the row 
$i=0$. The only step remaining is to verify that these solutions satisfy the respective boundary conditions
along the column $j=0$. This can be verfied by using \eqref{rec1} and \eqref{rec3} to get
\begin{equation*}
   \begin{pmatrix}
      f(i,j) & -g(i,0) \\
      f(i,0) & -g(i,j)
   \end{pmatrix}
   \begin{pmatrix}
      u_{i,j-1} \\
      u_{i-1,j-1}
   \end{pmatrix} 
= f(0,j) 
   \begin{pmatrix}
      u_{i,j} \\
      u_{i-1,j}
   \end{pmatrix}.
\end{equation*}
This gives us the condition
\begin{align*}
f(i+1,0) u_{i+1,j-1} = (f(i,j)+g(i+1,j))u_{i,j-1} - g(i,0) u_{i-1,j-1}. 
\end{align*}
Now it can be easily verified that the given boundary conditions are solutions of this for $j=1$
and the initial values $u_{1,0}$ and $u_{0,0}$ agree with our previous construction.
\end{proof}

Call the solutions to \eqref{rec1} corresponding to the initial conditions 
(a) and (b) of Prop. \ref{prop1}
as $q_{i,j}$ and $p_{i,j}$ respectively. Now we explore the arithmetic properties of 
these tables of rational numbers. Here we use an additional property of $f(i,j)$, namely
\begin{align}\label{cond3}
 f(0,x),f(x,0) \in \{ x^3, -x^3 \}.
\end{align}
We shall use the notation $d_n=[1,2,\ldots n]$ for the rest of this note.
\begin{prop}\label{prop2}
 Any solution $u$ of \eqref{rec1} has the property that 
 $u_{i,j}$ is a $\mathbb{Z}$-linear combination of $u_{i-1,j}$, $u_{i,j-1}$ and $u_{i-1,j-1}$ for $i,j\geq 1$.
Hence $q_{i,j}$ and $d_{\max(i,j)}^3 p_{i,j}$ are integers.  
\end{prop}
\begin{proof}
 For the proof of the first statement of the proposition, we start with the assumption that the gcd $(i,j)=1$.
 By \eqref{cond3}, this means that the gcd
$(f(0,j),f(i,0))=1$. Hence, there exists an integer $x$ such that
\begin{align*}
 f(i,j) \equiv x f(0,j) \mod f(i,0).
\end{align*}

Multiplying this by $g(i,j)$, we get
\begin{align*}
 f(i,j)g(i,j) &\equiv x f(0,j)g(i,j) \mod f(i,0).
\end{align*}

Using \eqref{cond1} makes the left side $\equiv f(0,j)g(0,j)$ and cancelling $f(0,j)$
 from both sides gives
\begin{align*}
 g(0,j) \equiv x g(i,j) \mod f(i,0).
\end{align*}

 Hence, by subtracting $x$ times the second row from the first row in \eqref{rec1}, we get 
coefficients which are divisible by $f(i,0)$. Thus we conclude that $u_{i,j}-x u_{i,j-1}$
is a $\mathbb{Z}$-linear combination of $u_{i-1,j}$ and $u_{i-1,j-1}$. Since $x$ is an integer,
we get the first part of the proposition, for the special case $(i,j)=1$.

\medskip

The general case $(i,j)=d$ is handled by calling $i'=i/d$, $j'=j/d$, dividing
\eqref{rec1} by $d^3$ and using the 
homogenity of $f(i,j)$ and $g(i,j)$ to get
\begin{equation*}
   \begin{pmatrix}
      f(i',j') & g(0,j') \\
      f(0,j') & g(i',j')
   \end{pmatrix}
   \begin{pmatrix}
      u_{i-1,j} \\
      u_{i-1,j-1}
   \end{pmatrix} 
= f(i',0) 
   \begin{pmatrix}
      u_{i,j} \\
      u_{i,j-1}
   \end{pmatrix}.
\end{equation*}
This reduces to the previous case as $(i',j')=1$ and we proceed as before and complete
the proof of the 
first statement of the proposition.

\medskip

For the second statement of the proposition, we use the first statement recursively to obtain that
 $u_{i,j}$ is a $\mathbb{Z}$-linear combination of 
$u_{0,0},u_{0,1},\ldots u_{0,j},u_{1,0},u_{2,0},\ldots u_{i,0}$.
Thus the second statement on
$q_{i,j}$ and $p_{i,j}$ follows from the arithmetic properties of the boundary values (a) and (b) in 
Prop. \ref{prop1}.
\end{proof}

Now, we will prove that $p_{i,j}/q_{i,j}$ converge to $\zeta(3)$ uniformly in $i$ and $j$. 
For simplicity of presentation we shall use one more proerty of $f,g$ namely
\begin{align}\label{cond4}
 f(i,j)-f(0,j) > g(i,j)-g(0,j), \qquad i,j\geq 1.
\end{align}

We define the table $\epsilon_{i,j}$ as
\begin{equation}\label{eps_def}
 \epsilon_{i,j}= q_{i,j} \zeta(3) - p_{i,j}.
\end{equation}
  and prove
\begin{prop}\label{prop3}
The table $\epsilon_{i,j} \rightarrow 0$ uniformly as $i,j \rightarrow \infty$. 
\end{prop}
\begin{proof}
 First we note that condition \eqref{cond4} implies that for $x\geq 1$,
\begin{align*}
  f(i,j) x +g(0,j) > f(0,j)x + g(i,j).
\end{align*}
Using this in \eqref{rec1} with $x=q_{i-1,j}/q_{i-1,j-1}$, gives the implication
$q_{i-1,j} \geq q_{i-1,j-1} \Rightarrow q_{i,j} > q_{i,j-1}$. Hence, we conclude that
$q_{i,j}$ is monotonically increasing along rows $i\geq 1$.

Define the following quantities: $$\delta_{i,j}^{row}=p_{i-1,j}q_{i-1,j-1}-p_{i-1,j-1}q_{i-1,j}, \qquad 
\delta_{i,j}^{col}=p_{i,j-1}q_{i-1,j-1}-p_{i-1,j-1}q_{i,j-1}.$$

 Now, we take the second row of \eqref{rec1} for both $p_{i,j}$ and $q_{i,j}$,
 multiply by $q_{i-1,j-1}$ and $p_{i-1,j-1}$ resp. and subtract to get
\begin{align}\label{prop3-1}
   j^3 \delta_{i,j}^{row}=i^3 \delta_{i,j}^{col}
\end{align}
 Similarly using the second rows of \eqref{rec3} gives
\begin{align}\label{prop3-2}
   -j^3 \delta_{i+1,j}^{row}=-i^3 \delta_{i,j}^{col}
\end{align}
Using \eqref{prop3-1} and \eqref{prop3-2} recursively with the initial values
$\delta_{1,j}^{row}=j^{-3}$ (verified directly) gives us 
 $\delta_{i,j}^{row}= j^{-3}$ and $\delta_{i,j}^{col}= i^{-3}$ for $i,j\geq 1$.

Now, we deduce the difference in the ratios $r_{i,j}=p_{i,j}/q_{i,j}$ along the columns,
\begin{align*}
 r_{i,j}-r_{i-1,j} = \frac{\delta_{i,j+1}^{col}}{q_{i,j}q_{i-1,j}} = 
\frac{1}{i^3 q_{i,j}q_{i-1,j}},
\end{align*}
tends to $0$ for $i$ fixed and $j\rightarrow \infty$. This, coupled with the fact that 
$r_{0,j}\rightarrow
\zeta(3)$ (from the boundary conditions in Prop. 1), implies that each row in $r_{i,j}$ has $\zeta(3)$ as limit.

As for the difference in $r_{i,j}$ along the rows,
\begin{align*}
 r_{i,j}-r_{i,j-1} = \frac{\delta_{i+1,j}^{row}}{q_{i,j}q_{i,j-1}} = 
\frac{1}{j^3 q_{i,j}q_{i-1,j}}.
\end{align*}
 Hence for any $i,j$ we have
\begin{align*}
 |\zeta(3)-r_{i,j}| \leq \frac{1}{q_{i,j}^2} \sum_{k\geq j}\frac{1}{k^3}.
\end{align*}
Multiplying by $q_{i,j}$ on both sides, we get that $|\epsilon_{i,j}|\leq \zeta(3)/q_{i,j}$
which tends to zero uniformly in $i,j$ by montonically increasing integers $q_{i,j}$ (see Prop. 2).
\end{proof}

\section{Proof of the theorem}

For the proof of the theorem we shall use the criterion that a number $\alpha$ 
is irrational if there exists a sequence of integers $a_n$ and $b_n$ such that
\begin{align}\label{criterion}
 0 \neq |a_n - b_n \alpha| \rightarrow 0  \quad \textrm{as } n \rightarrow \infty 
\end{align}
For, if not, let $\alpha=r/s$ with coprime integers $r,s$. Then the modulus of the $\mathbb{Z}$-linear
form in $1$ and $\alpha$ in \eqref{criterion} is either $0$ or $\geq 1/s$ 
for any choice of integers $a_n$ and $b_n$. This contradicts \eqref{criterion} and hence $\alpha$ is irrational.

The linear forms in $1$ and $\zeta(3)$ needed to use the above criterion will come from
the diagonal $\epsilon_{n,n}$. For estimating the decay of these forms, we shall need
Poincar\'e's theorem in vector form. For a discussion of the history of this theorem
 and references, see Aptekarev et. al. \cite{Aptekarev}, Ch.3, Sec.1.
Poincar\'e's theorem is usually
used in the simpler setting of $1$-dimensional
 recurrence relations and the following (\cite{Aptekarev},
 pp.1104) is a generalisation.
\begin{prop}\label{prop-P} \emph{(Poincar\'e-Perron, in vector form)}
 Let $\underbar{x}^n=(x_1^n,x_2^n,\dots x_k^n)$ be a sequence of 
$k$-dimensional vectors which is a solution of 
\begin{align}\label{diff-eqn}
 \underbar{x}^n = A_n \underbar{x}^{n-1},
\end{align}
where the vectors are taken to be column vectors and $A_n$ is a $k \times k$ matrix. Let 
$A_n \rightarrow A$ as $n\rightarrow \infty$, where $A$ is a diagonalizable matrix with 
eigenvalues of distinct magnitude. Then, either $\underbar{x}^n=0$ eventually or there exists
a component $j$ of $\underbar{x}^n$ such that
\begin{align}\label{limit}
 \lim_{n\rightarrow \infty} \frac{x_j^n}{x_j^{n-1}} = \lambda,\quad \textrm{and}
\quad \lim_{n\rightarrow \infty} \frac{\underbar{x}^n}{x_j^n} = \underbar{e},
\end{align}
 where $(\lambda,\underbar{e})$ is an eigenpair of $A$.

If the system \eqref{diff-eqn} is nondegenerate ($A_n$ is nonsingular for all $n$), then
for any eigenpair $(\lambda,\underbar{e})$ of $A$, there exists a solution $\underbar{x}^n$
of \eqref{diff-eqn} and component $j$ such that \eqref{limit} holds.
 \end{prop}

We use Prop.\ref{prop-P} by defining $\underbar{x}^n = (\epsilon_{n,n+1},\epsilon_{n,n})$ and
\begin{align}\label{A_n}
 A_n =    \begin{pmatrix}
      \frac{6 n^3 +9 n^2 +5 n + 1}{(n+1)^3} & \frac{- n^3}{(n+1)^3} \\
      1 & 0
   \end{pmatrix}
   \begin{pmatrix}
      6 & -1 \\
      1 & 0
   \end{pmatrix}.
\end{align}

We note that the difference equation \eqref{diff-eqn} is satisfied. This is because, 
by the definition \eqref{eps_def} of $\epsilon_{i,j}$
 it is also a solution of the recurrence \eqref{rec1} of Prop. 1. The second matrix on the right in \eqref{A_n}
is from \eqref{rec1} with $i=n,j=n$ which transforms $(\epsilon_{n-1,n},\epsilon_{n-1,n-1})$
to $(\epsilon_{n,n},\epsilon_{n,n-1})$. The first matrix is from \eqref{rec2} with
$i=n+1,j=n$ which transforms $(\epsilon_{n,n},\epsilon_{n,n-1})$
to $(\epsilon_{n,n+1},\epsilon_{n,n})$ as required.  

\medskip

We have $A_n\rightarrow A$ where 
\begin{align*}
 A = \begin{pmatrix}
      35 & -6 \\
      6 & -1
   \end{pmatrix},
\end{align*}
with the eigenvalues of $A$ being $17 \pm 12 \sqrt{2}$ and the corresponding eigenvectors
being $(18 \pm 12 \sqrt{2}, 6)$.

Now, we apply Prop. \ref{prop-P} and note that the system is nondegenerate for $n\geq 1$. Hence, the 
eventually zero case cannot occur because the vector $(\epsilon_{1,2},\epsilon_{1,1})$
is not the zero vector. Thus, there exists $j$ and an eigenpair $(\lambda,\underbar{e})$
such that \eqref{limit} holds.

Note that in \eqref{limit}, we can always choose $j$ to be any of the choices where $\underbar{e}_j\neq 0$.
Since neither of the eigenvectors of $A$ has any
 zero entries, we can conclude that 
\eqref{limit} holds
for both values $j=1,2$. We choose $j=2$ and get
\begin{align}\label{eps_decay}
 |\epsilon_{n,n}| = e^{n(\log \lambda + o(1))},
\end{align}
 where $\lambda = 17 \pm 12 \sqrt{2}$. However, $\lambda  = 17 + 12 \sqrt{2} $ is not 
possible because that will contradict Prop. \ref{prop3}. Hence, 
$\lambda  = 17 - 12 \sqrt{2}$ in the decay estimate \eqref{eps_decay}.

Recall that the prime number theorem implies that $d_n = e^{n(1+o(1))}$ (\cite{vdP}, p.198). From 
Prop. \ref{prop2}, we obtain that $d_n^3 p_{n,n}$ and $q_{n,n}$ are integers.
Hence, the decay estimate for $d_n^3 \epsilon_{n,n}$ is
\begin{align}\label{final_decay}
 |d_n^3 q_{n,n} \zeta(3) - d_n^3 p_{n,n}|  = e^{n(3 + \log \lambda + o(1))}.
\end{align}

Since $\log( 17-12\sqrt{2}) < -3$, we conclude that
 the $\mathbb{Z}$-linear form in $1$ and $\zeta(3)$ in \eqref{final_decay} 
is nonzero and tends to $0$ as 
$n \rightarrow \infty$. By the criterion \eqref{criterion},
we conclude that $\zeta(3)$ is irrational.

\section{Concluding remarks}

\noindent
\textit{Comparison with Ap\'ery's approach:}

At this point, we would like to stress that the construction of the tables $p_{i,j}$ and $q_{i,j}$
is entirely from \cite{Ap2} and it is Ap\'ery's method that we have adapted for the proof of
Prop. 1. The only novelty in our approach is in viewing these tables as the
solutions of a $2$-dimensional recurrence, whereas in \cite{Ap2} they appear as a sequence of
 continued fractions (indexed by row).

\medskip

 The advantage for us is that the proof of the arithmetic properties
in Prop. 2 is natural and elementary, whereas in \cite{Ap2}, a proof
for $q_{i,j} \in \mathbb{Z}$ is given using differential equations and
 there is no indication of the proof for
 the properties of $p_{i,j}$. This is a reasonable achievement for our approach,
given that these properties are considered deep (``a fundamental miracle'' \cite{vdP}, p. 202).
 Proposition 3 is also an easy consequence of the recurrence \eqref{rec1}, while there is not even a
 mention of this fact in \cite{Ap2}.

\medskip

The last step in \cite{Ap2} is the claim that the ``diagonal'' sequences $p_{i,i}$ and $q_{i,i}$ 
satisfy the recurrence \eqref{Ap-rec} and hence the (usual) Poincar\'e's theorem gives 
the irrationality of $\zeta(3)$. We remark that his claim can easily be verified
from our recurrence relations. We have preferred to follow the vector based approach
because it is more natural and the vector form of the Poicare\'e's theorem is not more difficult
than the $1$-dimensional version (see \cite{Aptekarev}).

\medskip
\noindent
\textit{Relation to Ramanujan's continued fraction:}

Now we sketch a proof that Ap\'ery's approach of viewing each row as a continued
 fraction recovers the intimate connection with Ramanujan's continued fraction \eqref{cont-fr}.
 First we note that along each row where $i$ is fixed, $p_{i,j}$ and $q_{i,j}$ satisfy \eqref{rec2} 
(with $i-1$
replaced by $i$). 
Hence the recurrence satisfied by $j!^3 p_{i,j}$ and $j!^3 q_{i,j}$ is 
\begin{equation}
 u_{i,j+1} = P(j,i) u_{i,j} -j^6 u_{i,j-1},
\end{equation}
for $j\geq 1$. Here we used the fact that $f(i+1,j)-g(i+1,j+1) = P(j,i)$ where 
$P(j,i)=j^3+(j+1)^3 +(4j+2)i(i+1)$ is the same polynomial appearing in \eqref{cont-fr}.   

Hence if we define the continued fraction $\omega(i)$ by
\begin{align*}
 \omega(i) = \frac{-1^6}{P(1,i)+} \, \frac{-2^6}{P(2,i)+}
\, \frac{-3^6}{P(3,i)+}\cdots
\end{align*}
then from the theory of continued fractions, we can obtain
\begin{align}\label{ram-last-step}
 \zeta(3) = \frac{\omega(i) p_{i,0} + p_{i,1}}{\omega(i) q_{i,0} + q_{i,1}}.
\end{align}
We recall that $q_{i,0}=1$ and $p_{i,0} = \sum_{n\leq i}n^{-3}$. It is also easy to obtain,
using the methods of Sec. 2, that $q_{i,1}=P(0,i)$ and $p_{i,1}=P(0,i)p_{i,0}+1$. Putting all
these values in \eqref{ram-last-step}, we get exactly the identity of Ramanujan \eqref{cont-fr} 
specialized at $x=i$. 

Now, in principle, Carlson's theorem can be used to derive \eqref{cont-fr}
for the entire halfplane $\mathrm{Re}\,x>-1/2$ from the equality at all positive integers,
if we show that the growth of the continued fraction is sufficiently slow
(the Hurwitz zeta function is bounded in this range).

\medskip
 
Thus, we see that the point of view in \cite{Ap2} gives a clear link with
\eqref{cont-fr} and bolsters the claim that Ap\'ery's constructions were indeed motivated from
\eqref{cont-fr}. We must mention at this point that the idea of using \eqref{cont-fr} to produce 
such a startling diophantine application is a testimony of the genius of Ap\'ery.
 
\medskip  
\noindent
\textit{Generalisations:}

As a final remark, these methods can be used as in \cite{Ap2} to show that $\log 2$ and $\zeta(2)$ 
are irrational 
by suitable choices for $f(i,j)$ and $g(i,j)$. For $\log 2$, the choice $f(i,j)=i+j$, $g(i,j)=i-j$ 
satisfies
all the conditions specified in Sec. 2 and yields a proof of irrationality as in Sec. 3. This is 
indeed related to continued fraction Entry 29 (Cor.) of \cite{Berndt}. Similarly for $\zeta(2)$,
the choice $f(i,j) = i^2 + ij+ \tfrac{1}{2} j^2$, $g(i,j) = -i^2 + ij- \tfrac{1}{2} j^2$ yields a 
proof of irrationality and is related to Entry 30 (Cor.) of \cite{Berndt}. Note that \eqref{cond3} is 
not satisfied in this case, and more work is needed to  
establish the analogue of Prop. 2. 

\medskip

However, there exist no nontrivial homogenous
polynomials $f(i,j)$, $g(i,j)$ of degree $>3$ which satisfy the conditions \eqref{cond1} and \eqref{cond2}. Since these 
conditions are crucial in the construction, we conclude that our approach, in its current
 form, fails for higher zeta values. One may need to look beyond $2 \times 2$ matrix recurrence relations
to attack these constants!

\medskip
\noindent
\textit{Acknowledgements}:
 I wish to thank Professors R. Balsubramanian, M. R. Murty, Y. Nesterenko, T. Rivoal, K. Srinivas and
 M. Waldschmidt for their suggestions and encouragement during various stages of this project.


\begin{thebibliography}{ASD} 

\bibitem{Ap1} R. Ap\'ery, Irrationalit\'e de $\zeta(2)$ et $\zeta(3)$, 
\textit{ Ast\'erisque} \textbf{61} 1979, 11--13.

\bibitem{Ap2} R. Ap\'ery, Interpolation de fractions continues et irrationalit\'e de certaines
constantes. \textit{C.T.H.S., Bull. de la Sect. des Sc. (Math.)} 
\textbf{III} Bibl. Nationale, Paris, 1981, 37--53.

\bibitem{Ap3} F. Ap\'ery, Roger Ap\'ery, 1916--1994: a radical mathematician.
\textit{Math. Intelligencer} \textbf{18} (2) 1996, 54--61.

\bibitem{Aptekarev} A. I. Aptekarev, V. I. Buslaev, A. Martínez-Finkelshtein and S. P. Suetin, 
Pad\'e approximants, continued fractions, and orthogonal polynomials \textit{Russ. Math. Surv.}
\textbf{66} (6) 2011, 1049--1131.

\bibitem{Berndt} B. C. Berndt, Ramanujan's Notebooks, Part II.
\textit{Springer-Verlag} New York, 1989.

\bibitem{Fis} S. Fischler, Irrationalit\'e de valeurs de z\^eta 
[d'apr\`es Ap\'ery, Rivoal, ...]. \textit{S\'eminaire Bourbaki} expos\'e num\'ero 910, 2002--2003.

\bibitem{vdP} A. van der Poorten, A proof that Euler missed... 
Ap\'ery's proof of the irrationality of $\zeta(3)$ (An informal report).
 \textit{Math. Intelligencer} \textbf{1} (4) 1978/79, 195--203.




\end{thebibliography}
\end{document}